\documentclass[UKenglish]{amsart}

\usepackage[utf8]{inputenx} 
\usepackage[T1]{fontenc}    

\usepackage{lmodern}           
\usepackage[scaled]{beramono}  
\usepackage[final]{microtype}  

\usepackage{amsmath}
\usepackage{amssymb}     
\usepackage{amsthm}    
\usepackage{thmtools}  
\usepackage{mathtools} 
\usepackage{mathrsfs}  

\usepackage{graphicx}  
\usepackage{babel}     
\usepackage{csquotes}  
\usepackage{textcomp}  
\usepackage{listings}  
\usepackage{subcaption}

\usepackage{varioref}
\usepackage{hyperref}
\usepackage[nameinlink, capitalize, noabbrev]{cleveref}

\declaretheorem[style = plain, numberwithin = section]{theorem}
\declaretheorem[style = plain,      sibling = theorem]{corollary}
\declaretheorem[style = plain,      sibling = theorem]{lemma}
\declaretheorem[style = plain,      sibling = theorem]{proposition}
\declaretheorem[style = definition, sibling = theorem]{definition}
\declaretheorem[style = definition, sibling = theorem]{example}
\declaretheorem[style = plain,      sibling = theorem]{conjecture}
\declaretheorem[style = remark,    numbered = yes]{remark}

\DeclareMathOperator{\Jac}{Jac}
\DeclareMathOperator{\di}{\delta}

\DeclareMathOperator{\ord}{ord}
\DeclareMathOperator{\Hess}{Hess}
\DeclareMathOperator{\trace}{tr}
\DeclareMathOperator{\adj}{adj}

\newcommand{\N}{\mathbb{N}}   
\newcommand{\C}{\mathbb{C}}   
\renewcommand{\P}{\mathbb{P}} 
\newcommand{\cA}{\mathcal{A}}
\newcommand{\cB}{\mathcal{B}}
\newcommand{\cC}{\mathcal{C}}

\newcommand{\cF}{\mathcal{F}}
\newcommand{\cG}{\mathcal{G}}
\newcommand{\cH}{\mathcal{H}}
\newcommand{\ol}{\overline}

\def\di{\partial}

\def\Ø{\emptyset}

\crefname{listing}{Program}{Programs}
\Crefname{listing}{Program}{Programs}

\usepackage{booktabs}

\begin{document}
\title[The 2-Hessian and Sextactic Points on Plane Algebraic Curves]{The 2-Hessian and Sextactic Points\\on Plane Algebraic Curves}
\author{Paul Aleksander Maugesten} 
\address{Department of Mathematics\\  University of Oslo\\ P.O. Box 1053 Blindern\\ NO-0316 Oslo\\ NORWAY} 
\email{paulamau@math.uio.no} 
\author{Torgunn Karoline Moe} 
\address{The Science Library\\ University of Oslo\\ P.O. Box 1063 Blindern\\ NO-0316 Oslo\\ NORWAY} 
\email{torgunnk@math.uio.no}
\begin{abstract}
 In an article from 1865, Arthur Cayley claims that given a plane algebraic curve there exists an associated 2-Hessian curve that intersects it in its sextactic points. In this paper we fix an error in Cayley's calculations and provide the correct defining polynomial for the 2-Hessian. In addition, we present a formula for the number of sextactic points on cuspidal curves and tie this formula to the 2-Hessian. Lastly, we consider the special case of rational curves, where the sextactic points appear as zeros of the Wronski determinant of the 2nd Veronese embedding of the curve.
\end{abstract}
\maketitle
\section{Introduction}\label{sec:intro}
Let $C=V(F)$ be an algebraic curve of geometric genus $g$ and degree $d$, given by a polynomial $F \in \C[x,y,z]_d$, in the projective plane $\P^2$ over $\C$. In standard terms, a point $p$ on an irreducible curve $C$ is called singular if all the partial derivaties of $F$ vanish, and smooth otherwise. Given two curves $C$ and $C'$ and a point $p$ in the intersection, let $(C \cdot C')_p$ denote the intersection multiplicity of $C$ and $C'$ at $p$. Moreover, for any point $p \in C$, let $m_p$ denote its multiplicity, i.e. the intersection multiplicity at $p$ of $C$ and a generic line. 

Now, given an irreducible curve $C$ and fixed $n \in \N$, consider curves, not necessarily irreducible, of degree $n$ in $\P^2$. With $r(n)=\frac{1}{2}n(n+3)$ and $nd \geq r(n)$, for every smooth point $p \in C$ there exists a curve of degree $n$ such that the local intersection multiplicity is equal to or bigger than $r(n)$, referred to as an \emph{osculating curve} to $C$ at $p$, see \cite{Arn96}. A smooth point where the intersection multiplicity between a curve of degree $n$ and $C$ is strictly bigger than $r(n)$ is referred to as an \emph{$n$-Weierstrass point}. In this case, the curve of degree $n$ is called a \emph{hyperosculating curve}. 

For $n=1$ this comes down to tangent lines to $C$, for each smooth point $p$ denoted by $T_p$ and given by the linear polynomial $xF_x(p)+yF_y(p)+zF_z(p)$, with the property that $l_p=(T_p \cdot C)_p\geq 2$. Thus, the smooth {$1$-Weierstrass} points on $C$, for which $l_p>2$, are nothing but the \emph{inflection points}. In particular, if $l_p=3$, $p$ is called a \emph{simple} inflection point, and, in general, the order of an inflection point is given by $v_p=l_p-2$.

The main focus of this paper is the case $n=2$. For every smooth point $p$ on a curve $C$ of degree $d \geq 3$ there exists a unique osculating curve of degree $2$, denoted by $O_p$, with $c_p=(O_p \cdot C)_p\geq5$ \cite{Cay59}. In particular, we look at $2$-Weierstrass points, where $c_p>5$, which includes inflection points and \emph{sextactic points}, first studied for curves of arbitrary high degree by Cayley in~\cite{Cay65}. 

In a broader context, Weierstrass points of curves in any projective space with respect to a linear system $Q$ have been intensively studied, both classically in the case of smooth curves, and more recently for singular curves \cite{BG97, Del08, Lak84,  LW90, Per90, Pie76}. One consequence of this research is that for a singular plane curve and a linear system $Q$, the $Q$-Weierstrass points include both the smooth $Q$-Weierstrass points and the singular points \cite{BG97}. 

Back in $\P^2$, in the case $n=1$ it is well known that the \emph{Hessian curve} to $C$, of degree $3(d-2)$, given by the polynomial
\[H=H_1(F)=
\begin{vmatrix}
F_{xx} & F_{xy} & F_{xz} \\
F_{yx} & F_{yy} & F_{yz} \\
F_{zx} & F_{zy} & F_{zz} \\
\end{vmatrix},
\]
intersects $C$ in its inflection points and singular points, i.e. its $1$-Weierstrass points. 

In \cite{Cay65} Cayley presents a curve with similar properties: a curve of degree $12d-27$ that intersects $C$ in its sextactic points, higher order inflection points, and its singular points. The first main result of this article is a correction of Cayley's defining polynomial for such a curve, referred to as the \emph{$2$-Hessian} to $C$. See \vref{sec:2hessian} for notation. 
\begin{theorem}[The $2$-Hessian]
	\label{thm:2hessian}
	Let $C=V(F)$ be a plane curve of degree $d \geq 3$, with $H$ the defining polynomial of the Hessian curve to $C$. Then there exists a curve of degree $12d-27$ given as the zero set of
	\begin{align*}
	H_2=H_2(F) = & \; (12d^2-54d+57)H \Jac(F,H,\Omega_{\bar H})\\
	 & \;+ (d-2)(12d-27)H \Jac(F,H,\Omega_{\bar F}) \\
	 & \;-\mathbf{20}(d-2)^2 \Jac(F,H,\Psi),
	\end{align*}
	such that the intersection points between $C$ and this curve are the singular points, the higher order inflection points, and the sextactic points of $C$.
\end{theorem}
\noindent Abusing notation we refer to both the Hessian curve and its defining polynomial as $H$, and similarly we write $H_2$ for both the $2$-Hessian curve and its defining polynomial.

Observe that the curve $V(H \cdot H_2)$ intersects $C$ in its $2$-Weierstrass points.

The situation is more complex for $n \geq 3$ and for curves in higher dimensional spaces. We refer to \cite{Cuk97} for modern results on higher Hessians to smooth plane curves and, in higher dimensional spaces, generalized Hessians to smooth curves that are complete intersections.     

\medskip
\noindent To complete the picture for $n=2$, we turn to another angle. Instead of studying sextactic points on a curve using its defining polynomial and the $2$-Hessian, we count the number of sextactic points using simple invariants of the curve, its inflection points, and its singular points. 

Recall that singular points on plane curves exist in many different shapes. Moreover, a singular point can be described by its multiplicity, delta invariant and number of branches. In particular, a singular point is called a \emph{cusp} if it is unibranched, and up to topological type, a cusp $p$ can be described by its multiplicity sequence, $\ol{m}_p$, the ordered sequence of multiplicities of the points above $p$ in the partial minimal embedded resolution of $C$ at $p$, see \cite[p.~503]{BK86}.

Curves with only cusp singularities, \emph{cuspidal curves}, have been thoroughly studied the last 25 years, see \cite{Moe08, Moe13} for an overview. The research has been motivated by both classification problems and the connections such curves have to problems in the theory of open surfaces. We mention briefly the Flenner--Zaidenberg rigidity conjecture, the Orevkov--Piontkowski conjecture, and the recently proved Coolidge--Nagata conjecture \cite{KP17}. 

The second main result in this article is a formula for the number of sextactic points on cuspidal curves, for which the invariants involved are easily accessible. Indeed, for a cusp $p$ with multiplicity $m_p$, there exists a unique line $T_p$ such that $l_p=(T_p \cdot C)_p>m_p$, referred to as the tangent line to $C$ at $p$. Similarly, applicable only to cusps where $l_p=2m_p$, there exists a, not necessarily unique, conic $O_p$ with $c_p=(O_p \cdot C)_p>2m_p$ and $c_p \neq3m_p,4m_p$, referred to as an osculating conic to $C$ at $p$, see \vref{lem:oscint}.

\begin{theorem}[Sextactic point formula]
\label{thm:SPF}
Let $C$ be a cuspidal curve of genus $g$ and degree $d \geq 3$. Let $I$ denote the set of inflection points and cusps on $C$ where $l_p \neq 2m_p$, and let $J$ denote the set of cusps on $C$ where $l_p=2m_p$. Then the number of sextactic points $s$ on $C$, counted with multiplicity, is given by
\begin{align*}
s = 6(2d + 5g - 5) - \sum_I (4m_p + 4l_p - 15)-\sum_J(10m_p+c_p-15).
\end{align*}
\end{theorem}

\begin{remark}
	Although the formula in \cref{thm:SPF} is a new restriction for cuspidal curves, it is superseded by other well known results.
\end{remark}

The formula in \cref{thm:SPF} has a classical touch and resembles the Pl\"ucker formula for inflection points; see \cite[p.~586]{BK86} for a formula for curves with any kind of singularities, attributed to Weierstrass and Noether.

Not surprisingly, there exist similar formulas for the number of sextacic points on plane curves, subject to different restrictions, in the literature. A formula for smooth curves is explicitly stated by Thorbergsson and Umehara in \cite[p.~90]{TU02}, and a formula for curves with certain singularties is given by Coolidge, attributed to Cayley, in \cite[Theorem 4, p.~280]{Coo31}. Both of these formulas coincide with the formula in \cref{thm:SPF} under the appropriate restrictions and generalizations. Moreover, as briefly commented by Coolidge in \cite[p.~280]{Coo31}, it is possible to find a formula for the number of sextactic points valid for any plane curve. Indeed, the formula in \cref{thm:SPF} can be directly extended to curves with multibranched singularities by dealing with one branch at the time. We do not state this formula explicitly, but we discuss some of the necessary generalizations in \vref{rem:gensing2}. Additionally, note that a method to compute the number of sextactic points on a plane cuve with any kind of singularities has been presented, with a very similar approach to ours, by Perkinson in \cite[pp.~42--44]{Per90}, and that a formula would also follow from such calculations.

For higher $n$, a formula for the number of $n$-Weierstrass points would require a stratification of the Weierstrass points and special treatments of each subgroup, as in the case of $I$ and $J$ in \cref{thm:SPF}. As described in \cite[p.~50]{Per90}, this would require close inspection of many points on the curve. Hence, a precise formulation of such a general formula is not within the scope of this article.

\medskip
\noindent This article has the following structure. In \cref{sec:2hessian} we fix Cayley's polynomial for the 2-Hessian to the one given in \cref{thm:2hessian}, and we explore some curves and their sextactic points using this tool.

In \cref{sec:sextactic} we prove the formula from \cref{thm:SPF} for the number of sextactic points on cuspidal curves, and we apply this formula to examples. Moreover, we derive a corollary which ties the sextactic point formula to the $2$-Hessian. 

In \cref{sec:rational} we take a closer look at rational curves. In this case, the osculating conic to a curve at a smooth point can be calculated directly from the parametrization. Moreover, we show that the zeros of the Wronski determinant of the 2nd Veronese embedding of the curve correspond to its $2$-Weierstrass points.

The figures in this article are made with \texttt{GeoGebra} \cite{GG}, computations are done with \texttt{Maple} \cite{Maple}, and program code can be found in \cite[Appendix B]{Mau17}.

\section{The $2$-Hessian curve}\label{sec:2hessian}
A fair bit of notation is necessary to present the defining polynomial of the $2$-Hessian to a curve. For the convenience of the reader, we stay close to Cayley's original formulations from \cite{Cay59,Cay65} and begin this section by recalling the essential objects.

With $C$ and $F$ as before, and $p$ a point on the curve, let 
\begin{align*}
DF_p(x,y,z) & = x F_x(p) + y F_y(p) + z F_z(p), \\
D^2F_p(x,y,z) & = x^2 F_{xx}(p) + y^2 F_{yy}(p) + z^2 F_{zz}(p) \\ & \quad + 2xy F_{xy}(p) + 2xz F_{xz}(p) + 2yz F_{yz}(p).
\end{align*} 

Moreover, with $\Hess(F)$ denoting the Hessian matrix of $F$, write  
\[
\Hess(F)=\begin{bmatrix}
a & h & g \\
h & b & f \\
g &  f &  c \\
\end{bmatrix},
\]
where as usual 
\[
a = F_{xx},\; b = F_{yy},\; c = F_{zz},\; f = F_{yz},\; g = F_{xz},\; h = F_{xy}, 
\] and $H=\det \Hess(F)$.

Similarly, for the Hessian matrix of $H$, write
\[
\Hess(H)=\begin{bmatrix}
a' & h' & g' \\
h' & b' & f' \\
g' &  f' &  c' \\
\end{bmatrix},
\]
so that the elements denote the second order partial derivatives of $H$,
\[
a' = H_{xx},\; b' = H_{yy},\; c' = H_{zz},\; f' = H_{yz},\; g' = H_{xz},\; h' = H_{xy}. 
\]

Following the pattern, the elements of the adjoint matrix $\Hess(F)^{\adj}$ are assigned the notation
\[
\Hess(F)^{\adj}=\begin{bmatrix}
\cA & \cH & \cG \\
\cH & \cB & \cF \\
\cG &  \cF &  \cC \\
\end{bmatrix},
\]
where
\[
\begin{array}{c}
\cA  = bc-f^2,\quad \cB = ac-g^2,\quad \cC = ab - h^2, \\ \cF = hg-af,\quad \cG= hf-bg,\quad \cH = fg-hc.
\end{array}
\]

Furthermore, put
\begin{align*}
\Omega & = (\cA,\cB,\cC,\cF,\cG,\cH) \cdot (a',b',c',2f',2g',2h'),
\end{align*}
which also can be expressed as the trace of the product of the adjoint matrix of the Hessian matrix of $F$ and the Hessian matrix of $H$: \[\Omega=\trace \left(\Hess(F)^{\adj} \cdot \Hess(H)\right).\] 

Additionally, let
\begin{align*}
\di_x\Omega_{\bar{H}} & = (\cA_x,\cB_x,\cC_x,\cF_x,\cG_x,\cH_x)\cdot(a',b',c',2f',2g',2h'), \\
\di_x\Omega_{\bar{F}}  & = (\cA,\cB,\cC,\cF,\cG,\cH)\cdot(a'_x,b'_x,c'_x,2f'_x,2g'_x,2h'_x).
\end{align*}
 Observe that the partial derivative of $\Omega$ with respect to $x$ can be written \[\Omega_x = \di_x\Omega_{\bar{H}} +\di_x\Omega_{\bar{F}},\] and note that none of the terms on the right hand side is an actual derivative. Naturally, similar expressions can be found for $\Omega_y$ and $\Omega_z$ by replacing $x$ with $y$ and $z$ in the above. Thus, although they are not true Jacobi determinants, we use Cayley's notation and write \[\Jac(F,H,\Omega_{\bar{H}})=\begin{vmatrix}
 F_x & F_y & F_z \\
 H_x & H_y & H_z \\
 \di_x \Omega_{\bar{H}} &  \di_y \Omega_{\bar{H}} &  \di_z \Omega_{\bar{H}} \\
 \end{vmatrix}\;\] and \[\Jac(F,H,\Omega_{\bar{F}})=\begin{vmatrix}
 F_x & F_y & F_z \\
 H_x & H_y & H_z \\
 \di_x \Omega_{\bar{F}} &  \di_y \Omega_{\bar{F}} &  \di_z \Omega_{\bar{F}} \\
 \end{vmatrix}.\]
 
Moreover, put
\[
\Psi  = (\cA,\cB,\cC,\cF,\cG,\cH) \cdot (H_x^2,H_y^2,H_z^2,2H_{y}H_{z},2H_{x}H_{z},2H_{x}H_{y}),
\]
or equivalently, in matrix form,
\begin{align*}
\Psi & = - 
\begin{vmatrix}
0   & H_x & H_y & H_z \\
H_x &   a   &   h   &   g   \\
H_y &   h   &   b   &   f   \\
H_z &   g   &   f   &   c
\end{vmatrix}.
\end{align*} Lastly, this time an actual Jacobi determinant, let $\Jac(F,H,\Psi)$ denote the determinant
\begin{align*}
\Jac(F,H,\Psi) = \begin{vmatrix}
F_x & F_y & F_z \\
H_x & H_y & H_z \\
\Psi_x & \Psi_y & \Psi_z \\
\end{vmatrix}.
\end{align*}    

\subsection{The osculating conic} 
\label{sub:the_OC}
For completion, before we move on to the $2$-Hessian, we present the osculating conic to a curve at a smooth point, give a formal definition of a sextactic point, and present an example.

For any smooth point $p$ on a curve $C$ of degree $d \geq 3$ there exists a unique osculating curve of degree 2. In the case of inflection points, this curve is the double tangent line. For a smooth point that is not an inflection point, the osculating conic was first presented by Cayley in \cite{Cay59}.

\begin{theorem}[\protect{\cite[p.~377]{Cay59}}]
	\label{thm:osccon}
	Let $C$ be a plane curve of degree $d \geq 3$ given by a polynomial $F$. If $p$ is a point on $C$ that is neither singular nor an inflection point, then, with $9H^3\Lambda = -3 \Omega H + 4 \Psi$, the osculating conic $O_p$ to $C$ at $p$ is given by
	\begin{align*}
	D^2F_p - \left(\tfrac{2}{3} \frac{1}{H(p)}{DH}_p + \Lambda(p) DF_p\right)DF_p=0.
	\end{align*}

\end{theorem}

\begin{definition}
	\label{def:sextactic_pt}
	On a plane curve $C$, a smooth point $p$ that is not an inflection point is called a \emph{sextactic point} if
	\[c_p=(O_p \cdot C)_p \geq 6.\] 
	With $s_p = c_p - 5$, a sextactic point $p$ is said to be of order $s_p$, or \emph{$s_p$-sextactic}.
\end{definition}

A first example of a sextactic point on a curve can be found on a cubic. 
\begin{example} \label{ex:nodalcubic}
	Let $C$ be the cubic curve with a nodal singularity given by \[F=y^2z-x^3-x^2z.\] 

Pick a smooth point on $C$, say $p_1=(\tfrac{-4}{5} : \tfrac{4}{5\sqrt{5}} : 1)$. The osculating conic $O_{p_1}$, for which $(O_{p_1} \cdot C)_{p_1}=5$, is shown in \vref{fig:osccon}. Its defining polynomial can be directly computed with the formula in \cref{thm:osccon} and \cite[Program B.2, pp.~67--68]{Mau17}, \[1125x^2+625y^2+64z^2+400\sqrt{5}yz+1200xz+350\sqrt{5}xy=0.\]
	
Now, pick $p_2=(-1:0:1)$. The osculating conic $O_{p_2}$ can be computed as above, \[2x^2+y^2+z^2+3xz=0.\] Note that $(O_{p_2} \cdot C)_{p_2}=6$, hence $p_2$ is a sextactic point and $O_{p_2}$ a hyperosculating conic, see \vref{fig:sextactic}.
	
	\begin{figure}[ht]
		  \centering
			\includegraphics[width=8cm]{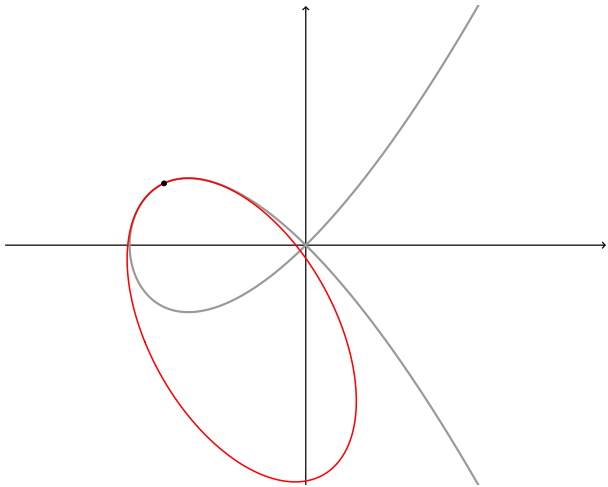}
		\caption{The nodal cubic from \cref{ex:nodalcubic} with the osculating conic $O_{p_1}$ at the smooth point $p_1$.}
		\label{fig:osccon}
\end{figure}	
\begin{figure}[ht]
		\centering
			\includegraphics[width=8cm]{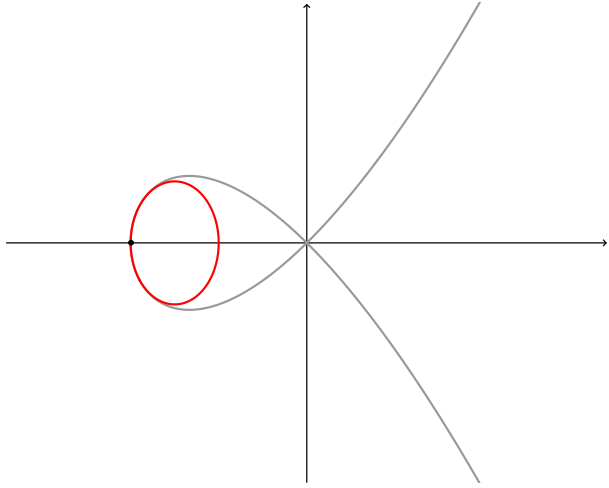}
		\caption{The nodal cubic from \cref{ex:nodalcubic} with the hyperosculating conic $O_{p_2}$ at the sextactic point $p_2$.}
			\label{fig:sextactic}
	\end{figure}
\end{example}

\begin{remark}
For a thorough investigation of a nodal cubic with respect to its sextactic points using elementary methods, we refer to the Appendix by Sakai in \cite[pp.~383--385]{Ton12} by Tono. Perhaps somewhat surprisingly, a sextactic point on such a curve plays an important role in the construction of a particularly interesting series of unicuspidal curves, first studied by Orevkov in \cite{Ore02}.
\end{remark}

\begin{remark}
From the perspective of abelian groups on curves, in \cref{ex:nodalcubic} it can be observed directly that $p_2=(-1:0:1)$ is a sextactic point. Indeed, the curve is given on Weierstrass form, and the vertical line $x+z=0$ is the tangent line to $C$ at $p_2$, passing through the identity element and inflection point $(0:1:0)$. Thus, $p_2$ is a $2$-torsion point for the abelian group structure on the smooth part of $C$, and it follows that $p_2$ is a sextactic point. 
\end{remark}

\subsection{The correct 2-Hessian}
An incorrect defining polynomial of the 2-Hessian to a curve $C$ is given by Cayley in \cite[p.~556]{Cay65}:
\begin{align*}
& (12d^2-54d+57)H \Jac(F,H,\Omega_{\bar H})\\
& + (d-2)(12d-27)H \Jac(F,H,\Omega_{\bar F}) \\
& - \mathbf{40}(d-2)^2 \Jac(F,H,\Psi).
\end{align*}
Cayley's proof of this result is based on restrictions that arise when more than five points in the intersection between a conic and the curve $C$ coalesce, and it is mostly correct. There is, however, an elementary computational error in the proof that affects the coefficient of the last term and makes the polynomial invalid. See \vref{thm:2hessian} for the correct defining polynomial of the $2$-Hessian to a curve. 

\begin{remark}
The mistake is neither corrected in the version of the paper \cite{Cay65} published in \emph{The collected mathematical papers} of Arthur Cayley \cite[p.~221]{Cay92}, nor in later works that cite the result, see \cite[p.~226]{Cuk97} and \cite[p.~372]{Sal79}. 
\end{remark}

\begin{proof}[Correction of Cayley's proof of \cref{thm:2hessian}]
	Cayley's mistake occurs in Section 19 on p.~553 in \cite{Cay65} as he attempts to simplify an expression that he has obtained for the $2$-Hessian in Section 10 on p.~550. 
	
	In the simplification Cayley introduces, in Section 18 on pp.~552--553, an expression $W$ and correctly states that 
	\begin{align*}
	W \coloneqq H\di\Omega - 5\Omega\di H = \frac{-3}{4d-9}\vartheta\Jac(F,\Omega,H) - \frac{5d-9}{4d-9}\di(\Omega H),
	\end{align*}
	where $\di= (F_y\nu -F_z\mu)\di_x + (F_z\lambda -F_x\nu)\di_y + (F_x\mu-F_y\lambda)\di_z$ and $\vartheta = \lambda x+\mu y+\nu z$, with $\lambda,\mu$, and $\nu$ arbitrary constants.

	Moreover, in Section 19, Cayley states that 
	\begin{align*}
\Psi \di H = \frac{{\bf{\frac{1}{2}}}}{4d-9} \vartheta \Jac(F,\Psi,H) + \frac{\frac{3}{2}(d-2)}{4d-9}H\di \Psi,
	\end{align*}
	and subsequently calculates 
	\begin{align}
	\label{eq3:2Hes-3}
	\begin{split}
	9HW + {\bf{40}}\Psi\di H = & -\frac{9(5d-9)}{4d-9}H\di(\Omega H) + \frac{60(d-2)}{4d-9}H\di \Psi \\ 
	& + \frac{\vartheta}{4d-9}\left[-27H\Jac(F,\Omega,H) + {\bf{40}} \Jac(F,\Psi, H)\right],
	\end{split}
	\end{align}
	where he makes the mistake of forgetting to multiply 40 by the $\frac{1}{2}$ in the first term of $\Psi\di H$. 

The correct calculations yield
	\begin{align}
	\label{eq3:2Hes-4}
	\begin{split}
	9HW+{\bf{40}}\Psi\di H = & -\frac{9(5d-9)}{4d-9}H\di(\Omega H) + \frac{60(d-2)}{4d-9}H\di \Psi \\ & + \frac{\vartheta}{4d-9}\left[-27H\Jac(F,\Omega,H) + {\bf{20}} \Jac(F,\Psi, H)\right],
	\end{split}
	\end{align}
	where the coefficient of $\Jac(F,\Psi,H)$ in the parenthesis is $20$ as opposed to $40$ in Equation \eqref{eq3:2Hes-3}.
	
	Following Cayley's remaining calculations in Sections 20--25 on pp.~553--556 in \cite{Cay65}, using the expression in \cref{eq3:2Hes-4}, leads to the polynomial in \cref{thm:2hessian}.
\end{proof}

\begin{remark}	
	Note that Cayley's result is stated for smooth curves in \cite{Cay65}. However, the proof is based on local considerations, hence the corrected polynomial identifies the higher order inflection points and sextactic points on singular curves as well. Moreover, each term in the polynomial $H_2(F)$ involves a determinant with the partial derivatives of $F$ in one row, so the $2$-Hessian certainly contains the singular points of $C$.  
\end{remark}

\begin{example} \label{ex:QuarticC1B}
	Let $C$ be the curve given by the defining polynomial \[F = x^4-x^3y+y^3z.\] This curve has one cusp with multiplicity sequence $\overline{m}_p=[3]$ and two simple inflection points. The defining polynomial for the 2-Hessian, computed with \cite[Program B.3, p.~69]{Mau17}, is  
	\begin{align*}
		-2^7\cdot3^{11}\cdot5\cdot7\cdot y^{18}\cdot\left(4x-y\right)\cdot \left(14x^2-7xy+2y^2\right)=0.
	\end{align*}
	
 The intersection points of $H_2$ and $C$ are
	\begin{align*}
	p_1 &= (0:0:1), \\
	p_2 &= \left(\tfrac{64}{3}:\tfrac{256}{3}:1\right), \\ 
	p_3 &= \left(\tfrac{49}{24}+i\tfrac{77\sqrt{7}}{24}:\tfrac{-637}{48}+i\tfrac{343\sqrt{7}}{48}:1\right), \\ 
	p_4 &= \left(\tfrac{49}{24}-i\tfrac{77\sqrt{7}}{24}:\tfrac{-637}{48}-i\tfrac{343\sqrt{7}}{48}:1\right).
	\end{align*}
	The point $p_1$ is the cusp, while $p_2$, $p_3$ and $p_4$ are sextactic points. With \cite[Program B.2, pp.~67--68]{Mau17} we compute the osculating conics for the latter and check with \texttt{Maple} \cite{Maple} that \[(O_{p_2} \cdot C)_{p_2}=(O_{p_3} \cdot C)_{p_3} = (O_{p_4}\cdot C)_{p_4} = 6.\] Note that Cayley's original formula for the $2$-Hessian identifies $p_2$ as a sextactic point, but not $p_3$ and $p_4$.
	
	A complete overview of this curve in terms of singularities, inflection points and sextactic points, and intersections with associated curves, can be found in \vref{tab:QuarticC1B}.
	
		\begin{table}[ht]
		\centering
		\resizebox{\textwidth}{!}{
			\begin{tabular}{ccccccc}
			
			Point $p$ & $\overline{m}_p$ & $\delta_p$ & \((T_p \cdot C)_p\) & \((O_p \cdot C)_p\) & \((H \cdot C)_p\) & \((H_2 \cdot C)_p\)\\
				\midrule
				$(0:0:1)$  
				& [3] & 3 & 4 & - & 22 & 81\\[1ex]
				
				$(8:16:1)$ 
				&  & 0 & 3 & - & 1  & 0 \\[1ex]
				
				$(0:1:0)$  
				&  & 0 & 3 & - & 1  & 0 \\[1ex]
				
				$(\frac{64}{3}:\frac{256}{3}:1)$ 
				&  & 0 & 2 & 6 & 0 & 1\\[1ex]
				
				$(\frac{49}{24}+i\frac{77\sqrt{7}}{24}:\frac{-637}{48}+i\frac{343\sqrt{7}}{48}:1)$ 
				&  & 0 & 2 & 6 & 0 & 1\\[1ex]
				
				$(\frac{49}{24}-i\frac{77\sqrt{7}}{24}:\frac{-637}{48}-i\frac{343\sqrt{7}}{48}:1)$ 
				&  & 0 & 2 & 6 & 0 & 1\\[1ex]
				
			\end{tabular}
		}
		\caption{Invariants and intersections for the rational cuspidal quartic in \cref{ex:QuarticC1B}.}
		\label{tab:QuarticC1B}
	\end{table}
\end{example}

\section{Sextactic point formulas}\label{sec:sextactic}
In this section we prove the formula for the number of sextactic points on a cuspidal curve in \vref{thm:SPF}, and we apply it to examples. Moreover, we state a corollary to the formula that reflects the intersection of $C$ with its Hessian and $2$-Hessian.

\subsection{Proof of the sextactic point formula}\label{pf:SPF}
The key ingredient in our proof is a generalized Pl\"ucker formula by Ballico and Gatto in \cite{BG97}. To each $Q$-Weierstrass point $p$ on a curve $C$ it is possible to assign a so-called $Q$-Weierstrass weight $w_p(Q)$. On the other hand, the sum of the Weierstrass weights can be computed through a generalization of the Brill--Segre formula. In our situation the result can be stated as follows.
\begin{proposition}[\protect{\cite[Proposition~3.4, p.~153]{BG97}}]
	\label{prop:BG34}
	Let $C$ be a projective, irreducible, cuspidal curve of geometric genus $g$, let $Q$ be a complete linear system of degree $\deg Q$ and dimension $r$, and let $w_p(Q)$ denote the Weierstrass weight of $C$ at $p$ with respect to $Q$. Then
	\begin{align*}
	\sum_{p\in C} w_p(Q) = (r+1)(\deg Q + rg - r).
	\end{align*}
\end{proposition} 

In the case of plane curves, to compute the $n$-Weierstrass weight $w_p(n)$ of a point $p \in C$ with respect to a complete linear system of curves of degree $n$ and dimension $r$, we use a technique by Notari from \cite[pp.~24--26]{Not99}. Assuming that $C$ is cuspidal, Notari's algorithm reduces to, for each point $p\in C$, finding curves $C_0,\ldots,C_r$ of degree $n$ such that the intersection multiplicities at $p$ are distinct. Subsequently, with $h_i = (C \cdot C_i)_p$, the $n$-Weierstrass weight of a unibranched point $p$ can be expressed as  
\begin{equation}
\label{eq:wp-used}
w_p(n) = \sum_{i=0}^r (h_i - i).
\end{equation}
Note that the integers $h_i$ do not constitute an ordered sequence.

The last important ingredient in the proof of \cref{thm:SPF} is \cref{lem:oscint}, which uses the Puiseux parametrization of $C$ at $p$ to determine $h_i$. 

\begin{remark}
In the remainder of this article we omit the index $p$ for the local invariants when only one point is discussed.
\end{remark}

\begin{lemma}\label{lem:oscint}
Let $p$ be a smooth point or a cusp of multiplicity $m$ on a plane curve $C$ of degree $d \geq 3$. With $T_p$ the tangent line at $p$, let $l=(T_p \cdot C)_p$. If $l\neq2m$, then a curve of degree $2$ intersects $C$ at $p$ with one of the following intersection multiplicities: 
\begin{align*}
h_0 = 0,\; h_1 = m,\; h_2 = l,\; h_3 = 2m,\; h_4 = m+l, \; h_5 = 2l.
\end{align*} 

If $l =2m$, then a curve of degree $2$ intersects $C$ at $p$ with one of the following intersection multiplicities:
\begin{align*}
h_0 = 0,\; h_1 = m,\; h_2 = 2m,\; h_3 = 3m,\; h_4 = 4m,\; h_5=c,
\end{align*}
where $c$ is subject to the restrictions \[ c > 2m \text{ and } c \neq 3m,4m.\] 
In particular, a curve of degree $2$ that intersects $C$ at $p$ with intersection multiplicity $c$ is irreducible. If $p$ is a cusp, such a curve is not necessarily unique, but $c$ is uniquely determined. 
\begin{proof}
Since $p$ is unibranched, the Puiseux parametrization of $C$ at $p$ can be given by \begin{equation*}
\label{eq:puiseux_param}
(t^m : at^l +\cdots : 1),
\end{equation*} where $a\neq0$, and ``$\cdots$'' denotes higher order terms in $t$ \cite[Cor.~7.7, p.~135]{Fis01}. 

\smallskip \noindent \textit{The case $l\neq2m$:}
We choose the standard basis for plane curves of degree 2, i.e.
	\[
		x^2,\; y^2,\; z^2,\; yz,\; xz,\; xy,
	\]
	and substitute the Puiseux parametrization of $C$ at $p$ into this basis. This gives
\[
x^2=t^{2m},\quad
y^2=a^2t^{2l} + \cdots,\quad
z^2= 1,\]\[
yz= at^l + \cdots,\quad
xz= t^m,\quad
xy= at^{m+l} + \cdots.
\]
	By assumption $l \neq 2m$, hence the basis elements represent curves with distinct intersection multiplicities at $p$, and taking the order of $t$ in each element provides the desired values. No other order is possible to obtain by any linear combination of the basis elements. 
	
	\smallskip \noindent \textit{The case $l=2m$:} 
	 Observe that since $d \geq 3$, the Puiseux parametrization has the form \[(t^m:at^{2m}+a_bt^{b}+\cdots:1)\] for some non-zero constants $a$ and $a_b$, with $b > 2m$, and we may assume that there are no terms of order between $2m$ and $b$ in the $y$-coordinate. Then
	\[
	x^2=t^{2m},\quad
	y^2= a^2t^{4m} + 2a\cdot a_b t^{2m+b}+ a_b^2t^{2b}+\cdots,\quad
	z^2=1,\]\[
	yz=at^{2m} + a_bt^b + \cdots,\quad
	xz= t^m,\quad
	xy= at^{3m} + a_bt^{b+m}+ \cdots.
	\]

Since $l = 2m$, two of the orders of $t$ in the basis elements are equal to $2m$, and \[h_0=0,\; h_1=m,\; h_2=2m,\; h_3=3m,\; h_4=4m.\] The remaining intersection multiplicity, $h_5=c$, can be found by computing the order of $t$ in all possible linear combinations of the basis elements. There are three main cases to consider. 

First, assume that $b \neq 3m,4m$. If $4m<b$, then \[yz-ax^2=0\] is the unique conic with order $c=b$ in $t$. If $3m<b<4m$, then any member of the family of conics \[yz-ax^2+k_1y^2=0,\] with $k_1 \in \C$, has order $c=b$ in $t$. If $2m<b<3m$, then any member of the family of conics \[yz-ax^2+k_1y^2+k_2xy=0,\] with $k_1,k_2 \in \C$, has order $c=b$ in $t$.

For the remaining two cases, $b=4m$ and $b=3m$, we need to ensure the existence of $c$. Indeed, if $b=km$ for some $k$, then since $p$ is unibranched, there exists a term with non-zero coefficient in the $y$-coordinate of the Puiseux parametrization such that the order of $t$ in this term is bigger than $b$, and $m$ is not a factor of this order. Thus, a linear combination of the basis elements will provide a curve of degree $2$ that intersects $C$ at $p$ with intersection multiplicity $c$ for some $c>b>2m$ and $c \neq 3m,4m$.

Thus, secondly, if $b=4m$, then \[yz-ax^2-\frac{a_b}{a^2}y^2=0\] is the unique conic with order $c$ in $t$ for a $c > b=4m$.

Third, if $b=3m$ and $c < 4m$, then any member of the family \[yz-ax^2-\frac{a_b}{a}xy+k_1y^2=0,\] with $k_1 \in \C$, has order $c$ in $t$. If $4m<c$, then \[yz-ax^2-\frac{a_b}{a}xy+\frac{a_b^2}{a^3}y^2=0,\] is the unique conic with order $c$ in $t$. 

In each of the above cases, the curves are irreducible and $c$ is uniquely determined, with $c>2m$ and $c \neq 3m,4m$.
\end{proof}
\end{lemma}

Motivated by the notion of tangent line at a cusp and \cref{lem:oscint}, we give the following definition.
\begin{definition}\label{def:oscconcusp}
An \emph{osculating conic} at a cusp $p$ for which $l=2m$ is a conic that intersects $C$ at $p$ with intersection multiplicity $c$, for a $c > 2m$ and $c \neq 3m,4m$.
\end{definition}

\begin{proof}[Proof of \cref{thm:SPF}]
	Let $C$ be a cuspidal curve of geometric genus $g$, and let $Q$ be the complete linear system of curves of degree 2 on $C$, with $\deg Q=2d$ and $r=\dim Q = 5$. Hence, the right hand side of the formula in \cref{prop:BG34} reads \[6(2d+5g-5).\]

For the left hand side of the formula, to compute the $2$-Weierstrass weight of a point $p$, we proceed by considering the set of points for which $l \neq 2m$ and the set of points for which $l=2m$, respectively.

\smallskip \noindent \textit{The case $l\neq2m$:} This case includes the inflection points and some of the cusps, and coincides with the set $I$.  

Inserting the values from \cref{lem:oscint} into \cref{eq:wp-used} yields
\begin{align*}
\begin{split}
w_p(2) & = \sum_{i=0}^5 (h_i - i) \\
& = 4m + 4l - 15.
\end{split}
\end{align*}

\smallskip \noindent \textit{The case $l=2m$:} Note that this case includes all smooth points that are not inflection points and some of the cusps; the latter points constitute the set $J$.

As above, \cref{lem:oscint} provides the possible intersections, and we gather
\begin{align*}
h_0 = 0,\; h_1 = m,\; h_2 = 2m,\; h_3 = 3m,\; h_4 = 4m,\; h_5=c.
\end{align*}
Thus, the $2$-Weierstrass weight of $p$ is
	\begin{equation*}
	\label{eq:leq2m}
	w_p(2)  = 10m+c-15.
	\end{equation*}

\noindent Note that if $p$ is smooth, then $w_p(2)$ is equal to its order as a sextactic point, $w_p(2)=c-5=s_p$. The formula is valid even when $p$ is not sextactic, as in this case $c=5$, or equivalently, $w_p(2)=0$. Hence, the total number of sextactic points on $C$, counted with multiplicity, is $s  = \sum s_p$.
	
	\medskip
	\noindent 
	Putting all this together while isolating $s$, we get
	\begin{align*}
	s = 6(2d + 5g - 5) - \sum_{I} (4m_p + 4l_p - 15) - \sum_{J} (10m_p + c_p - 15).
	\end{align*}
\end{proof}
\begin{remark}\label{rem:gensing2}
	A sextactic point formula for curves with arbitrary singularities follows from \cref{thm:SPF} by decomposing all singularities down to their irreducible branches, and subsequently calculating the Weierstrass weights for each branch separately, see \cite{Not99} and \cite{Per90}. In fact, the formula reads the same as in the cuspidal case, except that in this case $I$ denotes the set of inflection points and the branches of singular points where $l\neq 2m$, and $J$ denotes the branches where $l=2m$. In other words, both cusps, inflection points, and sextactic points might hide in branches of singular points with multiple branches, and these must be accounted for. 
\end{remark}

\subsection{A lemma and examples}
The essential ingredients depending on $C$ in the sextactic point formula are the degree $d$ and genus $g$, and $m_p$, $l_p$, and $c_p$ for its inflection points and cusps. Finding these numbers usually requires less heavy calculations than applying the $2$-Hessian directly. 

Moreover, we have the following lemma, which shows that $l_p$ and $c_p$ for a cusp $p$ sometimes can be determined merely by the degree of $C$ and the multiplicity sequence $\overline{m}_p$. 

\begin{lemma}\label{lem:boundc}
Let $p$ be a cusp with multiplicity sequence $\ol{m}=[m,m_1, \ldots, 1]$ on a plane curve $C$ of degree $d\geq 3$. Then \[d \geq l=km + m_k \geq m+m_1\] for some $k \geq 1$, with $m = m_1 = \ldots = m_{k-1}$. 

Moreover, if $l=2m$, \[2d \geq c=km + m_k > 2m\] for some $k \geq 2$, with $m = m_1 = \ldots = m_{k-1}$, and $c \neq 3m,4m$.

\begin{proof}
	In the first case, since $l$ is the intersection multiplicity of a curve and a line at a point, it follows from Bézout's theorem that $d \geq l$. Moreover, by \cite[Proposition~1.2, p.~440]{FZ96}, we have $l=km + m_k$ for some $k \geq 1$, with $m = m_1 = \ldots = m_{k-1}$, from which we derive the last inequality.

	In the second case, since $c$ is the intersection multiplicity of a curve and a conic at a point, it follows from Bézout's theorem that $c \leq 2d$. Additionally, by \cref{lem:oscint}, the curve $O_p$ is irreducible, so as above, \cite[Proposition~1.2, p.~440]{FZ96} ensures that $c=km + m_k$ for some $k \geq 1$, with $m = m_1 = \ldots = m_{k-1}$. Again by \cref{lem:oscint}, $c > 2m$ and $c \neq 3m,4m$, and the result follows.
	\end{proof}
\end{lemma}

In \cref{ex:QuinticC3B,ex:binomial} we consider two cuspidal curves and use the formula in \cref{thm:SPF} to compute the number of sextactic points on these curves. Note that these curves have the same degrees and singularities; they are equisingularly equivalent. However, the curves are not projectively equivalent, and they have different numbers of inflection points and sextactic points, see \cref{rem:3536}. We revisit these curves in \cref{ex:QuinticC3B2,ex:binomial2}.
\begin{example}\label{ex:QuinticC3B}
	Let $C$ be the cuspidal quintic given by \[F = y^5+2x^2y^2z - x^3z^2-xy^4.\] 
	By explicit calculations with appropriate associated curves it can be shown that this curve has two cusps; $p_1$ with multiplicity sequence $[3,2]$ and $p_2$ with multiplicity sequence $[2_2]$. Additionally, it has one simple inflection point, $p_3$, and two sextactic points, $p_4$ and $p_5$, see \cref{tab:QuinticC3B}. 
	
	 On the other hand, it is possible to compute the number of sextactic points directly using \cref{thm:SPF}. For $p_1$ with multiplicity sequence $[3,2]$, it follows from the first part of \cref{lem:boundc} that $l=3+2=5\neq6= 2m$, so \[w_{p_1}(2)=4\cdot 3 + 4 \cdot 5 - 15=17.\] For $p_2$ with multiplicity sequence $[2_2]$, \cref{lem:boundc} does not determine $l$, but it can be calculated directly from the defining polynomials of $C$ and $T_{p_2}$ that $l=4=2m$. By the second part of \cref{lem:boundc} we have that $c=2+2+1=5$, so \[w_{p_2}(2)=10\cdot 2+5-15=10.\] For the inflection point $p_3$, $m=1$ and $l=3$, so \[w_{p_3}(2)=4\cdot1+4\cdot3-15=1.\] 		
	
	Thus, the number of sextactic points on $C$ is
	\begin{align*}
	s  = & \;6 \cdot (2\cdot5 + 5 \cdot0 - 5) - 17-10-1\\
	 = & \;2.
	\end{align*}
	
		\begin{table}[ht]
		\centering
		\begin{tabular}{ccccccc}
				
			Point $p$ & $\overline{m}_p$ & $\delta_p$ & \((T_p \cdot C)_p\) & \((O_p \cdot C)_p\) & \((H \cdot C)_p\) & \((H_2 \cdot C)_p\)\\
			
			\midrule
			
			$(0:0:1)$  
			& [3,2] & 4 & 5 & - & 29 & 108 \\ [1ex]
			
			$(1:0:0)$  
			& [$2_2$] & 2 & 4 & 5 & 15 & 55 \\ [1ex]
			
			$(\frac{759375}{28672}:\frac{3375}{448}:1)$
			&  & 0 & 3 & - & 1 & 0 \\[1ex]
			
			$p_4$
			&  & 0 & 2 & 6 & 0 & 1 \\ [1ex]
		
			$p_5$
			&  & 0 & 2 & 6 & 0 & 1 \\ [1ex]
		
		\end{tabular}
		\caption[Quintic $C_{\text{3A}}$]{Invariants and intersections for the curve in \cref{ex:QuinticC3B}.}
		\label{tab:QuinticC3B}
	\end{table}
	
\end{example}

\begin{example}\label{ex:binomial}
	Let $C$ be the rational cuspidal quintic given by \[F=x^3z^2-y^5.\] This curve has two cusps, $p_1$ with multiplicity sequence $[3,2]$ and $p_2$ with multiplicity sequence $[2_2]$, and no inflection points. See \vref{tab:QuinticC3A}.	
  
  For both these cusps we have that $l=5 \neq 2m$, hence \[w_{p_1}(2)=4 \cdot 3+4 \cdot 5-15=17,\] and \[w_{p_2}(2)=4 \cdot 2+4 \cdot 5-15=13.\] 

Thus, the number of sextactic points on $C$ is
\begin{align*}
s  = & \;6 \cdot (2\cdot5 + 5 \cdot0 - 5) - 17-13\\
= & \;0.
\end{align*}

\begin{table}[ht]
	\centering
	\begin{tabular}{ccccccc}
		
		Point $p$ & $\overline{m}_p$ & $\delta_p$ & \((T_p \cdot C)_p\) & \((O_p \cdot C)_p\) & \((H \cdot C)_p\) & \((H_2 \cdot C)_p\)\\
		\midrule
		$(0:0:1)$  
		& [3,2] & 4 & 5 & - & 29 & 108 \\ [1ex]
		
		$(1:0:0)$  
		& [$2_2$] & 2 & 5 & - & 16 & 57 \\ [1ex]
		
	\end{tabular}
	\caption[Quintic $C_{\text{3A}}$]{Invariants and intersections for the curve in \cref{ex:binomial}.}
	\label{tab:QuinticC3A}
\end{table}
\end{example}

\begin{remark}\label{rem:3536}
	Note that the curves from \cref{ex:QuinticC3B,ex:binomial} both belong to the family of equisingular curves given by \[V\left(y^5-x(x z-\lambda y^2)^2\right),\quad \lambda \in \C,\] with $\lambda=1$ and $\lambda=0$, respectively. Indeed, for $\lambda\neq 0$, the curves in the family are algebraically equivalent to the curve in \cref{ex:QuinticC3B}. In this case, the intersection multiplicity of the curve and its tangent at the cusp with multiplicity sequence $[2_2]$ is equal to $4$, while in \cref{ex:binomial}, where $\lambda=0$, it jumps to $5$. The difference in the intersection multiplicities leads to different Weierstrass weights. This gives room for smooth Weierstrass points when $\lambda \neq 0$, while there are no smooth Weierstrass points when $\lambda = 0$, see also \cref{rem:binomial}.
	
\end{remark}

\subsection{A corollary that ties things together}
As a corollary to \cref{thm:SPF}, we state a formula that reflects the intersection of a curve of degree $d$ with its 2-Hessian of degree $12d-27$.
\begin{corollary}
\label{cor:SPF}
Let $C$ be a cuspidal curve of genus $g$ and degree $d \geq 3$, and let $\delta_p$ denote the delta invariant of a singular point $p$. Then with notation as in \cref{thm:SPF}, the following equations hold:
{\small{
\begin{align*}
d(12d - 27)+3d(d-2) &= s + 30 \sum \delta_p + \sum_I (4m_p+4l_p-15) + \sum_J(10m_p+c_p-15),\\
d(12d - 27) &= s + 24 \sum \delta_p + \sum_I (3m_p+3l_p-12) + \sum_J(7m_p+c_p-12).
\end{align*}
}}
\end{corollary}

\begin{remark}
	Recall that in the case of cusps, $\delta_{p}$ can be calculated from the multiplicity sequence $\ol{m}_p$, \[ \delta_{p}= \sum \frac{m_{i}(m_{i}-1)}{2},\] where $m_i$ is the $i$th element in $\ol{m}_p$. 
\end{remark}

Before we prove \cref{cor:SPF}, note that the two formulas could be interpreted as an application of Bézout's theorem to $C$ and its Hessian and 2-Hessian; the terms $3d(d-2)$ and $d(12d-27)$ are simply the product of the respective degrees. The remaining terms are local in nature, and we claim in \cref{con:local} that these terms reflect a natural geometrical interpretation. We have verified that the conjecture holds for all rational cuspidal curves of degree 4 and 5, see \cite{Mau17}. Note that a similar result is proved for the Hessian curve of a cuspidal curve, see \cite[Theorem 2.1.9, p.~32]{Moe13}.
\begin{conjecture}
\label{con:local}
The intersection multiplicity $(H_2 \cdot C)_p$ of a cuspidal curve $C$ and its 2-Hessian curve $H_2$ at a point $p$ is determined by the multiplicity $m$, the delta invariant $\delta$, and the intersection multiplicity with the tangent $l$ or the intersection multiplicity with an osculating conic $c$.

If $p$ is a point on $C$ such that $l\neq 2m$, then 
\begin{equation*}
(H_2 \cdot C)_p=24\delta+3m+3l-12.
\end{equation*}

If $p$ is a point on $C$ such that $l = 2m$, then
\begin{equation*}
(H_2 \cdot C)_p=24\delta+7m+c-12.
\end{equation*}
\end{conjecture}

\begin{proof}[Proof of \cref{cor:SPF}]
By substituting Clebsch' formula for the genus of a plane curve, see \cite[p.~393]{Har77}, \[g=\frac{(d-1)(d-2)}{2}-\sum \delta_p,\] into the formula from \cref{thm:SPF}, we infer that
\begin{equation}
\label{eq:sextclebsch}
s = 15d^2 -33d - 30 \sum \delta_p -\sum_I (4m_p + 4l_p - 15)-\sum_J(10m_p+c_p-15),
\end{equation} 
which can be rewritten as the first formula.

Moreover, the inflection point formula for cuspidal curves, explicitly stated in \cite[Theorem 2.1.8, p.~32]{Moe13}, reads \[v=3d(d-2)-\sum(6 \delta_p + m_p+l_p-3),\] where $v$ denotes the number of inflection points counted with multiplicity, and where the sum is taken over all cusps on $C$. This formula can be rewritten as 
\begin{equation}
\label{eq:inflpt}
0=3d^2-6d-6\sum \delta_p - \sum_{I \cup J}(m_p+l_p-3).
\end{equation} 
 
 By subtracting \cref{eq:inflpt} from \cref{eq:sextclebsch} and sorting terms, we reach the second formula.
\end{proof}

\section{Sextactic points on rational curves} 
\label{sec:rational}
In this section we assume that $C$ is a rational plane curve, i.e. $g=0$ and $C$ can be given by a parametrization \[\varphi(s,t)= (\varphi_0(s,t) : \varphi_1(s,t) : \varphi_2(s,t)), \; \text{for } (s:t) \in \mathbb{P}^1.\] 
Properties of this parametrization can be exploited to find key information about the $2$-Weierstrass points of a rational curve in a natural way. 

\begin{remark}
The results in \cref{thm:ratosccon,thm:ratww} build upon standard tools for studying Weierstrass points and hyperosculating spaces, see \cite{Arn96, Mir95, Per90, Pie76}, as well as results from the previous sections. Indeed, the classical flavour of the statements indicate that the results are well known. However, we have failed to find a suitable reference, and include the results and their proofs for completion.    
\end{remark}

First in this section, we state a corollary to \cref{thm:SPF} for rational cuspidal curves, which is obtained by setting $g=0$.
\begin{corollary}
	\label{cor:rational-SPF}
With notation as in \cref{thm:SPF}, the number of sextactic points $s$, counted with multiplicity, on a rational cuspidal curve of degree $d \geq 3$ is given by
	\begin{align*}
	s = 6(2d - 5) - \sum_I (4m_p + 4l_p - 15) -\sum_J(10m_p+c_p-15).
	\end{align*}
\end{corollary}

\subsection{The osculating conic for rational curves}
For a smooth point $p$ on a rational curve, it is possible to compute the osculating curve of degree $2$ directly from the parametrization. 
\label{sub:rational_osculating_conic}
\begin{theorem}
	\label{thm:ratosccon}
Let $C$ be a rational plane curve given by a parametrization $\varphi(s,t)$, and let $\omega(s,t)$ be the determinant
\begin{align*}
\omega(s,t)=\begin{vmatrix}
x^2 & y^2 & z^2 & yz & xz & xy \\[2pt]
\frac{\di^4 (\varphi_0^2)}{\di s^4} & \frac{\di^4 (\varphi_1^2)}{\di s^4} & \frac{\di^4 (\varphi_2^2)}{\di s^4} & \frac{\di^4 (\varphi_1\varphi_2)}{\di s^4} & \frac{\di^4 (\varphi_0\varphi_2)}{\di s^4} & \frac{\di^4 (\varphi_0\varphi_1)}{\di s^4} \\[5pt]
\frac{\di^4 (\varphi_0^2)}{\di s^3 \di t} & \frac{\di^4 (\varphi_1^2)}{\di s^3 \di t} & \frac{\di^4 (\varphi_2^2)}{\di s^3 \di t} & \frac{\di^4 (\varphi_1\varphi_2)}{\di s^3 \di t} & \frac{\di^4 (\varphi_0\varphi_2)}{\di s^3 \di t} & \frac{\di^4 (\varphi_0\varphi_1)}{\di s^3 \di t} \\[5pt]
\frac{\di^4 (\varphi_0^2)}{\di s^2 \di t^2} & \frac{\di^4 (\varphi_1^2)}{\di s^2 \di t^2} & \frac{\di^4 (\varphi_2^2)}{\di s^2 \di t^2} & \frac{\di^4 (\varphi_1\varphi_2)}{\di s^2 \di t^2} & \frac{\di^4 (\varphi_0\varphi_2)}{\di s^2 \di t^2} & \frac{\di^4 (\varphi_0\varphi_1)}{\di s^2 \di t^2} \\[5pt]
\frac{\di^4 (\varphi_0^2)}{\di s \di t^3} & \frac{\di^4 (\varphi_1^2)}{\di s \di t^3} & \frac{\di^4 (\varphi_2^2)}{\di s \di t^3} & \frac{\di^4 (\varphi_1\varphi_2)}{\di s \di t^3} & \frac{\di^4 (\varphi_0\varphi_2)}{\di s \di t^3} & \frac{\di^4 (\varphi_0\varphi_1)}{\di s \di t^3} \\[5pt]
\frac{\di^4 (\varphi_0^2)}{\di t^4} & \frac{\di^4 (\varphi_1^2)}{\di t^4} & \frac{\di^4 (\varphi_2^2)}{\di t^4} & \frac{\di^4 (\varphi_1\varphi_2)}{\di t^4} & \frac{\di^4 (\varphi_0\varphi_2)}{\di t^4} & \frac{\di^4 (\varphi_0\varphi_1)}{\di t^4}
\end{vmatrix}.
\end{align*}
Then, for a smooth point $p=\varphi(s_0,t_0)$, the polynomial $\omega(s_0,t_0) \in \C[x,y,z]_2$ is the defining polynomial of the osculating curve of degree 2 to $C$ at $p$.
\begin{proof}
Let $v_2(C) \subset \mathbb{P}^5$ denote the image of $C$ under the 2nd Veronese embedding of $\P^2$ to $\mathbb{P}^5$, such that \[v_2(C)(s,t)=(\varphi_0^2:\varphi_1^2:\varphi_2^2:\varphi_1\varphi_2:\varphi_0\varphi_2:\varphi_0\varphi_1).\] 

Now, consider the determinant $\tilde{\omega}(s,t)$, where in the first row of $\omega(s,t)$ the standard basis of plane conics is substituted with the coordinates of $\P^5$.

For a smooth point $v_2(C)(s_0,t_0)$, the linear polynomial $\tilde{\omega}(s_0,t_0)$ defines a unique osculating hyperplane to $v_2(C)$ in $\P^5$, and this hyperplane corresponds to the osculating curve of degree 2 to $C$ at $p=\varphi(s_0,t_0)$, with defining polynomial $\omega(s_0,t_0)$.
\end{proof}
\end{theorem}

Note that for an inflection point, $\omega(s_0,t_0)$ is reducible and equals the (double) tangent. For a smooth point that is not an inflection point, $\omega(s_0,t_0)$ and Cayley's osculating conic from \cref{thm:osccon} coincide by uniqueness.  
\begin{example}
The nodal cubic from \vref{ex:nodalcubic} can be given by the parametrization \[\varphi(s,t) = \left(st^2-s^3 : t^3-s^2t : s^3\right).\] 

At a smooth point $\varphi(s_0,t_0)$ the osculating curve of degree 2 has defining polynomial $\omega(s_0,t_0)$ equal to
	{\small
	\begin{align*}
&(2s_0^{10}+5s_0^8t_0^2+60s_0^6t_0^4+45s_0^4t_0^6)x^2+(s_0^{10}+10s_0^8t_0^2+5s_0^6t_0^4)y^2\\
	+\;&(s_0^{10}-5s_0^8t_0^2+10s_0^6t_0^4-10s_0^4t_0^6+5s_0^2t_0^8-t_0^{10})z^2
	-8  (5s_0^7t_0^3+6s_0^5t_0^5+5s_0^3t_0^7)yz\\
	+\;&(3s_0^{10}+70s_0^6t_0^4+40s_0^4t_0^6+15s_0^2t_0^8)xz -8(5s_0^7t_0^3+3s_0^5t_0^5)xy.
	\end{align*}}Evaluating this expression at the point $p = (-1:0:1)=\varphi(1,0)$ gives the same defining polynomial for $O_p$ as before,
	\begin{align*}
	2x^2+y^2+z^2+3xz=0.
	\end{align*}
	\end{example}
\subsection{The Weierstrass weight}
For rational curves, not necessarily cuspidal, information about its $2$-Weierstrass points can be found by direct computation and inspection of the zeros of a homogeneous determinantal polynomial. 
\begin{theorem} \label{thm:ratww}
	Let $C$ be a rational plane curve with parametrization $\varphi(s,t)$, and let $\xi(s,t)$ denote the Wronski determinant 
	\begin{align*}
	\xi(s,t)=\begin{vmatrix}
	\frac{\di^5 (\varphi_0^2)}{\di s^5} & \frac{\di^5 (\varphi_1^2)}{\di s^5} & \frac{\di^5 (\varphi_2^2)}{\di s^5} & \frac{\di^5 (\varphi_1\varphi_2)}{\di s^5} & \frac{\di^5 (\varphi_0\varphi_2)}{\di s^5} & \frac{\di^5 (\varphi_0\varphi_1)}{\di s^5} \\[5pt]
	\frac{\di^5 (\varphi_0^2)}{\di s^4 \di t} & \frac{\di^5 (\varphi_1^2)}{\di s^4 \di t} & \frac{\di^5 (\varphi_2^2)}{\di s^4 \di t} & \frac{\di^5 (\varphi_1\varphi_2)}{\di s^4 \di t} & \frac{\di^5 (\varphi_0\varphi_2)}{\di s^4 \di t} & \frac{\di^5 (\varphi_0\varphi_1)}{\di s^4 \di t} \\[5pt]
	\frac{\di^5 (\varphi_0^2)}{\di s^3 \di t^2} & \frac{\di^5 (\varphi_1^2)}{\di s^3 \di t^2} & \frac{\di^5 (\varphi_2^2)}{\di s^3 \di t^2} & \frac{\di^5 (\varphi_1\varphi_2)}{\di s^3 \di t^2} & \frac{\di^5 (\varphi_0\varphi_2)}{\di s^3 \di t^2} & \frac{\di^5 (\varphi_0\varphi_1)}{\di s^3 \di t^2} \\[5pt]
	\frac{\di^5 (\varphi_0^2)}{\di s^2 \di t^3} & \frac{\di^5 (\varphi_1^2)}{\di s^2 \di t^3} & \frac{\di^5 (\varphi_2^2)}{\di s^2 \di t^3} & \frac{\di^5 (\varphi_1\varphi_2)}{\di s^2 \di t^3} & \frac{\di^5 (\varphi_0\varphi_2)}{\di s^2 \di t^3} & \frac{\di^5 (\varphi_0\varphi_1)}{\di s^2 \di t^3} \\[5pt]
	\frac{\di^5 (\varphi_0^2)}{\di s \di t^4} & \frac{\di^5 (\varphi_1^2)}{\di s \di t^4} & \frac{\di^5 (\varphi_2^2)}{\di s \di t^4} & \frac{\di^5 (\varphi_1\varphi_2)}{\di s \di t^4} & \frac{\di^5 (\varphi_0\varphi_2)}{\di s \di t^4} & \frac{\di^5 (\varphi_0\varphi_1)}{\di s \di t^4} \\[5pt]
	\frac{\di^5 (\varphi_0^2)}{\di t^5} & \frac{\di^5 (\varphi_1^2)}{\di t^5} & \frac{\di^5 (\varphi_2^2)}{\di t^5} & \frac{\di^5 (\varphi_1\varphi_2)}{\di t^5} & \frac{\di^5 (\varphi_0\varphi_2)}{\di t^5} & \frac{\di^5 (\varphi_0\varphi_1)}{\di t^5}
	\end{vmatrix}.
	\end{align*}
	Moreover, let $(s_i:t_i)$ denote the distinct zeros of $\xi(s,t)$, with $i \leq 6(2d-5)$. Then the points $p_i=\varphi(s_i,t_i)$ are $2$-Weierstrass points on $C$, and the $2$-Weierstrass weight $w_{p_i}(2)$ is equal to the order of the zero of $\xi(s,t)$ corresponding to $(s_i:t_i)$.

\begin{remark}
	Note that the zeros in \cref{thm:ratww} correspond to all smooth $2$-Weierstrass points on $C$, but only the singular points with $w_p(2) > 0$, i.e. singular points where at least one branch is a cusp, an inflection point or a sextactic point.
	\end{remark}

\begin{proof}
First observe that whenever $\xi(s,t)$ vanishes, the corresponding point on $v_2(C)$ is either singular, or there exists a hyperplane in $\mathbb{P}^5$ that is hyperosculating to $v_2(C)$. As before, this hyperplane corresponds to a hyperosculating curve of degree 2 with respect to a point $p$ on $C$ in $\P^2$, hence determining an inflection point or a sextactic point.
 
For a smooth curve, \cite[VII.4, pp.~233--246]{Mir95} ensures that the multiplicity of a zero of $\xi(s,t)$ equals the $2$-Weierstrass weight of the corresponding point. Note that this is the same as the flattening points of the Veronese embedding, as described in \cite[p.~15]{Arn96}. This takes care of the smooth points. Alternatively, the below analysis for singular points can be applied to smooth points.
	
In the case of singular points, we consider each branch separately, and perform a local computation. So choose one branch and perform a linear transformation on $C$ so that the chosen branch of $p$ corresponds to the parameter values $(s:t)=(1:0)$, and so that its tangent is $y=0$. Moreover, by abuse of notation, observe that  	
\begin{align*}
\xi(1,t)=\xi(t)=\begin{vmatrix}
v_2(C)(t) \\
v_2(C)'(t) \\
v_2(C)''(t) \\
v_2(C)^{3}(t) \\
v_2(C)^{4}(t) \\
v_2(C)^{5}(t) \\
\end{vmatrix}.
\end{align*}

Assume first that the chosen branch can be parametrized by $(t^m:at^l+\ldots:1)$, with $a \neq 0$ and $l \neq 2m$. Substituting this into $\xi(t)$, computing the determinant, and comparing with the proof of \cref{thm:SPF} in \cref{pf:SPF}, it follows that \[\ord_t \xi(t) = 4m+4l-15= \sum_{i=0}^5(h_i-i).\] 

If $l = 2m$, first transform the branch of $C$ at $p$ so that it is given by the parametrization $(t^m:at^{2m}+\ldots:1)$, where $a \neq 0$. Subsequently, by applying the Veronese embedding, consider the curve \[\rho(t)=(t^{2m}:a^2t^{4m}+\ldots:1:t^m:at^{2m
}+\ldots:at^{3m}+\ldots)\] in $\P^5$, which by a suitable linear transformation in $\P^5$ can be given by the parametrization \[\sigma (t)=(t^{2m}:a^2t^{4m}+\ldots:1:t^m:a_ct^{c
}+\ldots:at^{3m}+\ldots),\] for an $a_c \neq 0$, and $c>2m$, $c \neq 3m,4m$, see \vref{lem:oscint}. Then, for a parametrized curve $\psi$ in $\P^5$, consider the determinant 
	\begin{align*}
W_{\psi}(t)=\begin{vmatrix}
\psi(t) \\
\psi'(t)\\
\psi''(t)\\
\psi^{(3)}(t)\\
\psi^{(4)}(t)\\
\psi^{(5)}(t)\\
\end{vmatrix}.
\end{align*}
Straightforward computations give that the order of $t$ in $W_{\sigma}(t)$ is $10m+c-15$. Now, $\ord_tW_{\sigma}(t)=\ord_tW_{\rho}(t)=\ord_t\xi(t)$, hence $\ord_t\xi(t)=10m+c-15$. Moreover, notice that in this case the inverse image of the hyperplane $x_4=0$ under the linear transformation in $\P^5$ and the Veronese embedding corresponds to a conic in $\P^2$ that intersects the branch of $C$ at $p$ with intersection multiplicity $h_5=c$. Hence, we have that $\ord_t \xi(t) = \sum_{i=0}^5(h_i-i)$.

Performing a similar analysis on all branches of $C$ at $p$, and summing up, we reach $w_p(2)$.
\end{proof}
\end{theorem}

\begin{remark}
Observe that the polynomial $\xi$ is homogeneous in $s$ and $t$ of degree $6(2d - 5)$. Since the $2$-Weierstrass weights add up to this number, \cref{thm:ratww} provides another proof of \cref{cor:rational-SPF}.
\end{remark}

 We now revisit \cref{ex:QuinticC3B,ex:binomial}, and compute the $2$-Weierstrass points and weights using the Wronski determinant in \cref{thm:ratww}.

\begin{example}\label{ex:QuinticC3B2}
	Let $C$ be the rational cuspidal quintic from \vref{ex:QuinticC3B}, with parametrization \[\varphi(s,t)=(s^5 : s^3t^2 : st^4+t^5).\]
	Computing the Wronski determinant gives
	\begin{align*}
	\xi(s,t) = -2^{24}\cdot3^{12}\cdot5^2\cdot7^4\cdot s^{17}t^{10}(192s^3 + 1680s^2t + 5275st^2 + 5250t^3).
	\end{align*}
The cusps $p_1$ and $p_2$ correspond to the parameters $(0:1)$ and $(1:0)$, respectively, while the inflection point $p_3$ and the sextactic points $p_4$ and $p_5$ correspond to zeros of $192s^3 + 1680s^2t + 5275st^2 + 5250t^3$. Determining the order of the corresponding zeros, we conclude, as in \cref{ex:QuinticC3B}, that $w_{p_1}(2)=17$, $w_{p_2}(2)=10$, and $w_{p_3}(2)=w_{p_4}(2)=w_{p_5}(2)=1$.
\end{example}

\begin{example}\label{ex:binomial2}
	Let $C$ be the rational cuspidal quintic from \vref{ex:binomial}, with parametrization \[\varphi(s,t)=(s^5 : s^3t^2 : t^5).\]	
	
	Computing the Wronski determinant gives
	\begin{align*}
	\xi(s,t) = -2^{25}\cdot 3^{13}\cdot 5^{5}\cdot 7^{5} \cdot s^{17}t^{13}.
	\end{align*}
The cusp $p_1$ corresponds to the parameter value $(0:1)$, and the exponent of $s$ gives the $2$-Weierstrass weight, $w_{p_1}(2)=17$. The cusp $p_2$ corresponds to $(1:0)$, thus $w_{p_2}(2)=13$. 
\end{example}

\begin{remark}\label{rem:binomial}
Note that the curve from \cref{ex:binomial,ex:binomial2} is an example of a cuspidal curve $C_{m,l}$ of degree $l$ with defining polynomial \[F=x^mz^{l-m}-y^l.\] For any $l,m \in \mathbb{N}$, with $m<l$ and $\gcd(m,l)=1$, the curve $C_{m,l}$ is cuspidal; bicuspidal when $1 <m <l-1$, and unicuspidal with an inflection point otherwise. It can be shown that these curves have no other $n$-Weierstrass points for $1 \leq n < l$. A proof in a more general setting can be found in \cite[Section 4]{BG97}, but, additionally, it is possible to construct a direct proof of this claim with methods from the present article, see \cite[Theorem 5.4.3, p.~54]{Mau17}.
\end{remark}

\section*{Acknowledgements}
This article is based on results from the master's thesis of the first author \cite{Mau17}, supervised by the second author. We would like to thank Geir Ellingsrud for his senior supervision and support of this paper, Kristian Ranestad for interesting discussions, and Ragni Piene, Georg Muntingh, the editor and the referee for valuable comments. Moreover, we want to express our gratitude to Live Rasmussen and the Science Library at the University of Oslo for making this project possible and for still having a printed edition of \emph{The collected mathematical papers} of Arthur Cayley in-house and available on a Saturday morning.

\bibliographystyle{amsalpha2}
\bibliography{bibMS}{}

\end{document}